\documentclass{amsart}
\usepackage{amssymb}
\usepackage[english]{babel}

 \newtheorem{thm}{Teorema}[section]
 \newtheorem{cor}[thm]{Corolary}
 \newtheorem{lem}[thm]{Lemma}
 \newtheorem{prop}[thm]{Proposition}
 \theoremstyle{definition}

 \newtheorem{exa}[thm]{Example}
 \newtheorem{defn}[thm]{Definition}
 \theoremstyle{remark}
 \newtheorem{rem}[thm]{Remark}
 \numberwithin{equation}{section}
 
 \DeclareMathOperator{\IM}{Im}

 \newcommand{\abs}[1]{\left\vert#1\right\vert}

\newcommand{\C}{{\mathbb C}}  \newcommand{\N}{{\mathbb N}}   \newcommand{\R}{{\mathbb R}}  \newcommand{\B}{{\mathbb B}}

\newcommand{\CB}{\mathcal B} \newcommand{\CF}{\mathcal F} \newcommand{\CH}{\mathcal H}

\newcommand{\CK}{\mathcal K} \newcommand{\CP}{\mathcal P}  \newcommand{\CW}{\mathcal W}

   \DeclareMathOperator{\Ker}{Ker}
   \DeclareMathOperator{\Aut}{Aut}

\begin{document}

\title{Parametrization of representations of braid groups}

\author{Claudia Mar\'{i}a Egea, Esther Galina }

\address{Facultad de Matem\'atica Astronom\'{i}a y F\'{i}sica, Universidad Nacional de C\'ordoba, C\'ordoba, Argentina}

\email{cegea@mate.uncor.edu, galina@mate.uncor.edu }

\thanks{This work was partially supported by CONICET, SECYT-UNC, FONCYT}

\begin{abstract}
 We give a method to produce representations of the braid group $\B_n$ of
$n-1$  generators ($n\leq \infty$). Moreover,  we give sufficient conditions over a non unitary representation for being of this type. This
method produces examples of irreducible representations of finite and infinite dimension. \end{abstract}


\maketitle \section{Introduction}\label{Introduccion}

The braid group of $n$ string, $\B_n$, is defined by generators and relations as follows $$ \B_{n}=<\tau_{1}, \dots,\tau_{n-1}>_{/\sim} $$

$$ \sim=\{ \tau_k\tau_j=\tau_j\tau_k, \textrm{ if }|k- j|>1;  \ \
        \tau_k\tau_{k+1}\tau_k=\tau_{k+1}\tau_{k}\tau_{k+1},  \ \ 1\leq k\leq n-2\}
$$

We will consider non unitary representations of the braid groups on separable Hilbert spaces, they are homomorphisms of groups
$\psi:\B_n\rightarrow \CB(\CH)$, where  $\CB(\CH)$ is the set of bounded linear operators on a Hilbert space $\CH$. In this context we will
admit $n=\infty$, where $\B_{\infty}$ is the braid group with infinity many generators.

One of the purpose of this work is to obtain non unitary representations of a locally compact group $G$ which satisfies the second axiom of
countability. Given a $5$-tuple $(\pi, X, \mu, \nu, U)$, where $(X, \mu, \nu)$ define a direct integral Hilbert space $\CH$, $\pi$ is an action
of $G$ on $X$, the measure $\mu$ and the function $\nu$ satisfy certain compatibility with the action $\pi$ and $U$ is a cocycle on $\CB(\CH$)
that verifies some relations (definition \ref{cocycle}), then we define a non unitary representation $\phi_{(\pi, X, \mu, \nu, U)}$ of $G$.

In the particular case of the braid group, we can also establish sufficient conditions  over a non unitary representation to be of this type.
That is, given a representation $\rho$ which satisfies certain conditions we present a $5$-tuples $(\pi, X, \mu, \nu, U)$ such that $\rho
\approx \phi_{(\pi, X, \mu, \nu, U)}$. Therefore, we parametrize a family of representations by $5$-tuples. This family is really huge.

A similar parametrization was obtained by Varadarajan through of {\it{systems of imprimitivity}}. The systematic development of this concept is
due to Mackey and he applies his results to the notion of induced representations. The reader can consult Chapter VI of \cite{V} for references
and comments. A system of imprimitivity is a pair $(\rho, \CP)$, where $\rho$ is a unitary representation over a separable Hilbert space $\CH$
and $\CP:=\{P_{\alpha}: \alpha\in A\}$ is a family of projections such that $$ \rho(g) P_{\alpha} \rho(g)^{-1}=P_{g\cdot \alpha}$$ Varadarajan
shows that given a system of imprimitivity there exist certain parameters, equivalents to our $5$-tuple, such that $\rho$ is like $\phi_{(\pi,
X, \mu, \nu, U)}$. But from the properties of a non unitary representation $\rho$ in the family that we parametrize, we propose a set of
projections $\CP$, that depends on $\rho$, such that $(\rho, \CP)$ can be a system of imprimitivity.

The family of representations that we consider is the following. Let $\psi$ be a representation of $\B_n$, $n\leq \infty$, over a separable
Hilbert space $\CH$ such that the family  $\CF=\{\psi(\tau_k)
 \psi(\tau_k)^* : k=1, \dots, n-1\}$ is commutative and each operator of $\CF$ has
 discrete spectral decomposition, that is
$$ \psi(\tau_k) \psi(\tau_k)^*=\sum_{l\in I_k} \lambda_{k,l} P_{k,l} $$ for some projections $P_{k,l}$ and complex numbers $\lambda_{k,l}$. Note
that if $\dim \CH< \infty$, the last condition is trivially satisfied. At infinite dimension, if the operator $\psi(\tau_k)$ is compact it is
satisfied.

Let $N$ be the von Neumann algebra generated by $\CF$. As $\CF$ is a commutative set, then $N$ is a commutative algebra and the space $\CH$ can
be seen as a direct integral Hilbert space (see section \ref{Notacion}). We also impose that $\psi(\tau_k) P_{j,l} \psi(\tau_k)^{-1}\in N$, for
all spectral projection $P_{j,l}$ associated to the operator $\psi(\tau_j) \psi(\tau_j)^*$. These relations say that there exists an action
$\pi$ of $\B_n$ on $N$. This action induce one of $\B_n$ on $\CH$ that permits us to give an expression for each $\psi(\tau_k)$ in terms of a
$5$-tuple $(\pi, X, \mu, \nu, U)$.

The strategy of using the theory of commutative von Neumann algebras to parametrize representations  has been used by G$\ddot{\rm{a}}$rding and
Wightman to parametrize the representations of commutation and anti-commutation relations \cite{GW1}, \cite{GW2}. In \cite{G} the reader can
find more details. It has  been also used by Galina, Kaplan and Saal to classify real, complex or quaternionc representations of the Clifford
algebras associated to infinity dimensional vector spaces \cite{GKS1}, \cite{GKS2}.

We give explicitly the $5$-tuple associated to some well known representations as the standard representation \cite{TYM} or \cite{S} and some
local representations \cite{AG}, \cite{AS} (the ones obtained from solutions of the Yang-Baxter equation, see section \ref{Ejemplos}). In the
same way, the finite dimensional representations given in \cite{EG} can be obtained by $5$-tuples.

This work is divided in $6$ sections. In the section $2$, we give some basic results and notation of von Neumann algebras. The main result of
this section is theorem \ref{teorema 1} that gives a method to obtain a non unitary representation of a locally compact group $G$ from a
$5$-tuple $(\pi, X, \mu, \nu, U)$. In section $3$, we prove the theorem that parametrize the representations of the braid group mentioned before
through of $5$-tuples, using the results of von Neumann algebras theory. In section $4$, we present the corresponding parameters of some well
known representations as the standard representation and diagonal local representations. In this way, we can generalize these classical
representations of $\B_n$ to representations of $\B_{\infty}$. In section $5$ we analyze some relations between the measure $\mu$ and the action
of the group $\B_n$ on $X$. In particular, for $n=\infty$, we show that some non discrete measure, as product measure or Lebesgue measure for
the interval $[0,1]$, can be considered. Finally, in section $6$ we give some conditions in the $5$-tuple to guarantee the irreducibility of the
associated representation.

\section*{Acknowledgment} The authors thanks to Nicol\'as Andruskiewitsch for helpful discussions.

\section{A way to obtain non unitary representations}\label{Notacion} In this section we first revise some concepts and we fix notation. We give
a method to obtain non unitary representations of a locally compact group $G$ from a specific data.

\begin{defn} Let $(X,\mu)$ be a measure space, that is, $X$ is a set with a $\sigma$-algebra of subsets and $\mu$ is a non-negative completely
additive set function defined on the $\sigma$-algebra. Let $\nu: X\rightarrow \N \cup \{\infty\}$ a measurable function, $\nu$ divides the set
$X$ into disjoints measurable subsets $X_m=\{x\in X :\nu(x)=m\}$, where
 $m\in \N \cup \{\infty\}$. For each $x\in X$, let $m:=\nu(x)$, and let $\CH_x:= \CH_m$ a fixed $m$-dimensional separable Hilbert space.

Let $\CH$ be the space consisting of measurable vector functions $f$, $x\rightarrow f(x)$, such that $f(x)\in \CH_x$ for almost all $x\in X$ and
$ \int_X \|f(x)\|^2 d\mu <\infty $. A vector function $f$ is measurable if for each $h\in \CH_m$ the scalar function $\alpha (x)=<f(x), h>$ is
measurable over $X_m$. In $\CH$ we consider the addition, multiplication by scalars and the scalar product respectively defined by $$
x\rightarrow f(x) +g(x), \ \ \ \ \  x\rightarrow \lambda f(x), \ \ \ \ \ <f,g>=\int_X <f(x), g(x)>d\mu(x). $$

Then $\CH$ is a Hilbert space called the {\it{Direct Integral of the Hilbert space $\{\CH_x\}$ over $(X,\mu)$}}. We write $\CH=\int_X^{\oplus}
H_x d\mu (x)$. \end{defn}

The function $\nu$ is called {\it{dimension function}}. Sometimes, we will identify the Hilbert space $\CH$ with the triple $(X, \mu, \nu)$.

\begin{exa} If $X=\N$ with the measure that assigns to each set its cardinal, the direct integral is just the direct sum of Hilbert space of the
family $\{\CH_n\}$. \end{exa}

\begin{exa} If $(X,\mu)$ is any measure space and $\CH_x=\C$ for each $x\in X$, then the direct integral is $L^2(X,\mu)$ the space of square
integrable functions on $X$. \end{exa}

\begin{defn}\label{diago} Let $\CH=\int_X^{\oplus} \CH_x d\mu (x)$. A bounded linear operator $T$ over $\CH$ is said to be {\it{decomposable}}
if there is a function $x\rightarrow T(x)$ over $X$ such that $T(x)$ is a bounded linear operator of $\CH_x$, and for each $f\in \CH$,
$(Tf)(x)=T(x) f(x)$ for almost every $x\in X$.

If, in addition, each $T(x)$ is a multiple of the identity operator of $\CH_x$, we say that $T$ is {\it{diagonalizable}}. \end{defn}

If $\CH=\int_X^{\oplus} H_x d\mu (x)$, for each essentially bounded measurable function $\varphi \in L^\infty(X,\mu) $, there is a
diagonalizable operator on $\CH$ given by $$ (M_{\varphi} f)(x)= \varphi (x) f(x) $$

The set of operators $M_{\varphi}$ form a weakly closed commutative $*$-subalgebra of $\CB(\CH)$.

\begin{defn} A {\it{von Neumann algebra}} is a weakly closed $*$-subalgebra  of $\CB(\CH)$ for some Hilbert space $\CH$, that contains the
identity operator of $\CH$. \end{defn}

With this definition we may rewrite the last statement in the following way. \begin{rem} \label{integral directa} If the Hilbert space $\CH$  is
the direct integral of the Hilbert spaces $\CH_x$, then the subalgebra $\{M_{\varphi}\}$ of $\CB(\CH)$ is a commutative von Neumann algebra.
Moreover, the converse holds too. For all commutative von Neumann algebra in a separable Hilbert space $\CH$, there is a measure space $(X,
\mu)$ and a measurable family of Hilbert spaces $\{\CH_x\}$ such that $\CH= \int_X^{\oplus} \CH_x d\mu(x)$ and each element of the algebra has
the form $M_{\varphi}$ for some essentially bounded function $\varphi$ (see by example \cite{vN} or theorem 14.2 of \cite{KR}). \end{rem} \

Let $G$ be a locally compact group. We want to define a representation of $G$ over $\CH$ the direct integral Hilbert space associated to $(X,
\mu, \nu)$. Let $\pi: G \times X \rightarrow X$ be an action of $G$ on $X$, that is, for each $g\in G$, $\pi(g)$ is a continuous function, and
$\pi(gh)=\pi(g) \pi(h)$ for all $g,h \in G$. Suppose that the measure $\mu$ and the dimension function $\nu$ verify the following
compatibilities with $\pi$. \begin{defn} The measure $\mu$ is said {\it{$\pi$-quasi-invariant}} if the measures $\mu(x)$ and $\mu(\pi(g)x)$ have
the same set of measure zero, for all $g\in G$. That is, $\mu(E)=0$ if and only if $\mu(\pi(g)E)=0$, for all $g\in G$.

If $\mu(\pi(g)E)=\mu(E)$ for all $E$ subset of $X$ and for all $g\in G$, then the measure $\mu$ is said {\it{invariant}}. \end{defn}

\begin{defn} The dimension function $\nu(x)$ is said {\it{$\pi$-invariant}} if $\nu(x)=\nu(\pi(g)x)$, for all $g\in G$, and for almost every
$x\in X$. \end{defn}

We need just one more ingredient. Given $\CH$ a Hilbert space, then $\CB(\CH)$ is a Borel space with the smallest $\sigma$-algebra of measurable
subsets which makes all the maps $A\rightarrow <Af, h>$  measurable, for all $f,h\in \CH$. \begin{defn}\label{cocycle} Let $G$ be a locally
compact group, and $\CH$ a direct integral Hilbert space associated to the triple $(X, \mu, \nu)$. Let $\pi$ be an action of $G$ on $X$, such
that $\mu$ is $\pi$-quasi-invariant and $\nu$ is $\pi$-invariant. A {\it{$(G,X,\CH)$-cocycle relative to $\mu$}} is a function $U:G\times
X\rightarrow \CB(\CH)$ which satisfies the following properties: \begin{enumerate} \item $U$ is a Borel map, \item for each $g\in G$, $U(g,.)$
is a decomposable bounded operator of $\CB(\CH)$, with decomposition function $x\rightarrow U(g,x)$, \item  $U(1,x)=1_{\CH} $, for almost all
$x\in X$, \item  $U(g_1g_2,x)=U(g_1,\pi(g_2)x)U(g_2,x)$ for almost all $x\in X$. \end{enumerate} \end{defn}


\begin{thm}\label{teorema 1} Let $G$ be a locally compact group and let $(\pi, X, \mu, \nu, U)$ be a $5$-tuple where \begin{enumerate} \item
$\pi$ is an action of $G$ on $X$; \item $(X, \mu)$ is a measure space with $\mu$ a $\pi$-quasi-invariant measure; \item $\nu: X\rightarrow \N
\cup \{\infty\}$ is a $\pi$-invariant measurable function; \item $U: G\times X\rightarrow \CB(\CH)$ is a cocycle relative to $\mu$, where $\CH$
is the direct integral Hilbert space associated to $(X, \mu, \nu)$. \end{enumerate} Then $$ \phi=\phi_{(\pi, X, \mu, \nu, U)}: G \rightarrow
\CB(\CH) $$ defined by \begin{equation} \label{representacion} (\phi(g) f)(x)=\sqrt{\frac{d\mu (\pi(g^{-1})x)}{d\mu(x)}}U(g,\pi(g^{-1})x)
f(\pi(g^{-1})x) \end{equation} is a representation of $G$ on $\CH$. \end{thm}

\begin{proof}   

\begin{equation} \label{repres} \phi(gh)=\phi(g) \phi(h) \end{equation}

By Lemma 5.28 of \cite{V}, $\phi$ is continuous if it is measurable. By definition of the $\sigma$-algebra of measurable subsets of $\CB(\CH)$,
it is enough to verify that for $f,h\in\CH$, the function $G\rightarrow \C$ defined by $<\phi(g)f,h>$ is measurable. But $$ \begin{aligned}
<\phi(g)f,h>&=\int_X <(\phi(g)f)(x), h(x)>d\mu(x)\\
            &=\int_X \sqrt{\frac{d\mu (\pi(g^{-1})x)}{d\mu (x)}}< U(g,\pi(g^{-1})x) f(\pi(g^{-1})x), h(x)>d\mu(x)
\end{aligned} $$ Since the integrand on the right is a Borel function on $G\times X$, because $U$ is a Borel function, and the function $\int_X
<\cdot, \cdot> d\mu(x)$ of $\CH \times \CH$ in $\R$ is a Borel function, $<\phi(g)f,h>$ is a measurable function.

Let us compute $\phi(gh)$, let $f\in \CH$ $$ \begin{aligned} (\phi(gh)f)(x)&=\sqrt{\frac{d\mu (\pi((gh)^{-1})x)}{d\mu (x)}}
U(gh,\pi((gh)^{-1})x) f(\pi((gh)^{-1})x) \\
              &=\sqrt{\frac{d\mu (\pi(h^{-1})\pi(g^{-1})x)}{d\mu (x)}}U(gh,\pi(h^{-1})\pi(g^{-1})x) f(\pi(h^{-1}) \pi(g^{-1})x)
\end{aligned} $$ On the other hand $$ \begin{aligned} (\phi(g) \phi(h) f)(x)&=\sqrt{\frac{d\mu (\pi(g^{-1})x)}{d\mu (x)}}U(g,\pi(g^{-1})x)
(\phi(h)f)(\pi(g^{-1})x) \\
                      &=\sqrt{\frac{d\mu (\pi(g^{-1})x)}{d\mu (x)}}\sqrt{\frac{d\mu
(\pi(h^{-1})\pi(g^{-1})x)}{d\mu (\pi(g^{-1})x)}}\\
                      &\qquad \qquad \qquad U(g,\pi(g^{-1})x)U(h,\pi(h^{-1})\pi(g^{-1})x) f(\pi(h^{-1}) \pi(g^{-1})x)
\end{aligned} $$ The conditions on $U$ to be a cocycle and the property of the measure $\mu$ imply that \ref{repres} is true. So, $\phi$ is a
representation of $G$ on $\CH$. \end{proof}

Note that two different tuples may define equivalent representations.

The operator $\phi(g)$ is not necessarily a unitary operator; it will be, if $U(g,.)$ is unitary. Varadarajan shows this results when the
function dimension $\nu$ is constant and the cocycle $U$ is unitary, that is $U(g,.)$ is a unitary operator for each $g\in G$ (see \cite{V},
theorem 6.7, pag 215).

 \

If $G$ is a discrete group presented by generators and relations, it is enough to define a cocycle $U$ in the generators such that they satisfy
the relations of $G$. In the case $G=\B_n$, the equations of cocycle for the generators are

\begin{equation} \label{ecuacion 1} \begin{aligned} U(\tau_k, & \pi(\tau_k^{-1})x) U(\tau_{k+1}, \pi(\tau_{k+1}^{-1})\pi(\tau_k^{-1}) x)
U(\tau_k, \pi(\tau_k^{-1})\pi(\tau_{k+1}^{-1})\pi( \tau_k^{-1}) x) = \\
   & = U(\tau_{k+1}, \pi(\tau_{k+1}^{-1})x) U(\tau_k,\pi(\tau_k^{-1})\pi(\tau_{k+1}^{-1})x)U(\tau_{k+1},\pi(\tau_{k+1}^{-1})\pi(\tau_k^{-1})\pi(
   \tau_{k+1}^{-1})x)
   \end{aligned}
\end{equation} if $1\leq k\leq n-2$ \begin{equation}\label{ecuacion 2} U(\tau_k, \pi(\tau_{k}^{-1})x)
U(\tau_j,\pi(\tau_{j}^{-1})\pi(\tau_{k}^{-1})x)= U(\tau_j, \pi(\tau_{j}^{-1})x) U(\tau_k, \pi(\tau_{k}^{-1})\pi(\tau_j^{-1})x) \end{equation}
 if $|j-k|>1$.

From now on fix the notation for $\pi_k x=\pi(\tau_k)x$. \begin{exa} \label{sencillas} Let $X=\{(x_1, \dots , x_n): x_i\in A\}$ for some set
$A\subset \N$, $n\leq \infty$, let $\pi$ the action of $\B_n$ on $X$ given by

$$ \pi_k (x_1, \dots , x_k,x_{k+1}, \dots, x_n)= (x_1, \dots ,x_{k+1}, x_k, \dots, x_n) $$

Note that $\pi_k^{-1}x =\pi_k x$. Then, the following equations on $U(\tau_k, \pi_k x)$ implies the equations \ref{ecuacion 1} and \ref{ecuacion
2}. They have a more simple expression.

\begin{enumerate} \item  $U(\tau_{k+1},\pi_{k+1} x)= U(\tau_k, \pi_k \pi_{k+1} \pi_k x)$ if $1\leq k\leq n-2$, \item  $U(\tau_k,\pi_k \pi_j x)=
U(\tau_k, \pi_j x)$ if $|j-k|>1$, \item  the operators $U(\tau_k, \pi_k x)$, $U(\tau_k, \pi_k\pi_{k+1} x)$, and $U(\tau_k, \pi_k \pi_{k+1} \pi_k
x)$ commute and $U(\tau_k, \pi_k x)$ commutes with $U(\tau_k, \pi_k \pi_j x)$. \end{enumerate}

Furthermore, in this case, $(1)$ permits us to define inductively the cocycle in the following way, we define $U(\tau_1, x)$ for all $x\in X$.
Then, we obtain $U(\tau_k, x)$, $2\leq k\leq n-1$, from $(1)$. Finally we verify the other conditions. Therefore this gives a machinery to
construct examples. \end{exa}

\begin{exa} Let $\rho$ be a representation of $\B_n$ on a vector space $V$. Consider the $5$-tuple $(\pi, X, \mu, \nu, U)$ likewise in Theorem
\ref{teorema 1}, where $\nu(x)=\dim V$ for almost all $x\in X$ and the cocycle is a constant function on $X$ defined by $U(\tau_k,x)=
\rho(\tau_k)$. Then we can generalize any representation of $\B_n$. In section \ref{irreduc} we are going to give conditions in the $5$-tuple
such that the associated representation to be irreducible.

Note that if $X=\{x_0\}$, $\mu(x_0)=1$ and $\pi$ is the trivial action of $\B_n$ on $X$, then $\phi_{(\pi, X, \mu, \nu, U)}\approx \rho$. But
this is a trivial parametrization of representations. In the next section, we are going to give a non trivial parametrization. \end{exa}

In section \ref{Ejemplos} we will give more interesting examples.

\section{Parametrization of a family of representations of $\B_n$} \label{teoFundamental}

With the notation of the previous section, we characterize a class of representations of $\B_n$ through the  $5$-tuple $(\pi, X, \mu, \nu, U)$
parametrization.

Let $\psi: \B_n \rightarrow \CB(\CH)$ be a representation of the group $\B_n$ ($n\leq \infty$), over a separable Hilbert space $\CH$. We denote
by $\psi_k:=\psi(\tau_k)$. Assume that the family $\{\psi_k \psi_k^*\}$ is a commutative family of operators with discrete spectral
decomposition, that is \begin{equation} \label{condicion (1)} \psi_k \psi_k^*=\sum_{l\in I_k} \lambda_{k,l} P_{k,l} \textrm{ \ \  for all  }  k
\end{equation}

\begin{equation} \label{condicion (2)} \psi_k \psi_k^* \psi_j \psi_j^*=\psi_j \psi_j^* \psi_k \psi_k^*, \textrm{\ \   for all  } j, k
\end{equation}

Assume that the cardinality of $I_k$ is greater than 1 for technical reasons. In particular, $\psi_k \psi_k^*$ is a non unitary operator. The second condition ensures
that the von Neumann algebra $N$ generated by the family $\widetilde{\CF}=\{\psi_k \psi_k^*\}$ is commutative. Then, by remark \ref{integral
directa}, $\CH$ decomposes in a direct integral associated to a triple $(X', \mu, \nu)$. Respect of this decomposition, each element of $N$
corresponds to a diagonalizable operator $M_{\varphi}$ of $\CB(\CH)$. In particular, each projection is exactly the operator multiplication by
the characteristic function $\chi_{X_{k,l}}$ of a measurable set $X_{k,l}$ of positive measure. That is $$ P_{k,l}=M_{\chi_{X_{k,l}}} $$

It is known that every von Neumann algebra is generated by its spectral projections. Then $N$ is generated by the spectral projections
$P_{k,l}$. So,  it is enough to consider the $\sigma$-algebra of subsets of $X'$ generated by $X_{k,l}$ with $l\in I_k$ and $k=1,\dots, n-1$.
Furthermore, since $\sum_{l\in I_k} P_{k,l}=1_{\CH}$ for each $k$, we have that $\bigcup_{l\in I_k} X_{k,l}=X'$, then each $x'\in X'$ belongs to a
set $X_{k, x_k}$ for some $k\in \{1,\dots, n-1\}$. Hence, we can replaced $X'$ by its factor set $X$, whose points are the subset of $X'$ of the
form $\cap_{k=1}^{n-1} X_{k,x_k}$. Thus $\CH_x= \overline{\cap_{k=1}^{n-1} \IM P_{k, x_k}}$.

If we assign the number $x_k$ to each set $X_{k,x_k}$, then each element $x\in X $ can be identified with the $(n-1)$-tuple $(x_1, \dots ,
x_{n-1})$ ($n$ can be $\infty$). Under this identification note that $X_{j,l}=\{(x_1, \dots, x_{n-1})\in X : x_j=l\}$. Hence, $X_{j,l}\subset X$
and the $\sigma$-algebra de subsets of $X$ and $X'$ coincide.

If the dimension of $\CH$ is finite, this reduction to the diagonal form is not more than the process of the simultaneous diagonalization of a
commutative family of operators.

On the other hand,  $\psi_k P_{j,l} \psi_k^{-1}$ is a projection, assume that it is in $N$ and denote it  by $P_{\pi_k(j,l)}$, that is $$
P_{\pi_k(j,l)}:= \psi_k P_{j,l} \psi_k^{-1} $$ Or equivalently, since $P_{j,l}=M_{\chi_{X_{j,l}}}$, we have that

\begin{equation}\label{condicion (3)}
\psi_k M_{\chi_{X_{j,l}}} \psi_k^{-1}= M_{\chi_{\pi_k(X_{j,l})}}
\end{equation}

We can define an action $\pi$ of the group $\B_n$ in the $\sigma$-algebra of the subsets of $X$ by $\pi(\tau_k):=\pi_k$, where the values in the
generators of the $\sigma$-algebra are $$ \begin{array}{lcc} \           &     \pi_k    & \\
 X_{j,l}   & \longrightarrow & \pi_k(X_{k,l})
 \end{array}
$$ Indeed, by the braid group equations we have that $$ \psi_k \psi_{k+1} \psi_k P_{j,x_j}\psi_k^{-1} \psi_{k+1}^{-1}\psi_k^{-1}=
\psi_{k+1}\psi_k \psi_{k+1}P_{j, x_j} \psi_{k+1}^{-1} \psi_k^{-1} \psi_{k+1}^{-1} $$ if $1\leq k\leq n-2$ and $$ \psi_i \psi_k P_{j, x_j}
\psi_k^{-1} \psi_i^{-1}= \psi_k \psi_i P_{j, x_j} \psi_i^{-1} \psi_k^{-1} $$ if $|i-k|>1$,  hence $$ \pi_k \pi_{k+1} \pi_k=\pi_{k+1} \pi_k
\pi_{k+1} $$ for all $k$ such that $1\leq k\leq n-2$, and $$ \pi_k \pi_j =\pi_j \pi_k $$ if $|j-k|>1$. So, the braid group $\B_n$ acts in the
$\sigma$-algebra of subsets of $X$.

\begin{lem}\label{accion} This action of $\B_n$ in the $\sigma$-algebra of sets of $X$ induces an action
 on $X$ given by
$$ \pi_k(x)=\pi_k(\cap_{j=1}^{n-1} X_{j,x_j}):=\cap_{j=1}^{n-1} \pi_k(X_{j,x_j}) $$ \end{lem} \begin{proof} We must see that $\cap_{j=1}^{n-1}
\pi_k(X_{j,x_j})\neq\O$ and that $\pi_k(x)$ is an element of $X$, this is $\cap_{j=1}^{n-1} \pi_k(X_{j,x_j})= \cap_{j=1}^{n-1} X_{j,l_j}$ for
some $l_j \in I_j$.

In fact, $x$ is identified with the set $\cap_{j=1}^{n-1} X_{j,x_j}$ which is associated to the non-zero projection $\wedge_{j=1}^{n-1}
P_{j,x_j}$ (indeed, if $P$ and $Q$ are projections, $P\wedge Q$ denotes the orthogonal projection onto $\overline{\IM P \cap \IM Q})$. In the
same way, $\cap_{j=1}^{n-1} \pi_k(X_{j,x_j})$ is associated to the operator $\wedge_{j=1}^{n-1}( \psi(\tau_k) P_{j,x_j} \psi(\tau_k^{-1}))=
\psi(\tau_k)( \wedge_{j=1}^{n-1} P_{j,x_j}) \psi(\tau_k^{-1})$. It is non zero since $\psi(\tau_k)$ is invertible. Hence, $\cap_{j=1}^{n-1}
\pi_k(X_{j,x_j})\neq\O $.

Note that for each $j$ and $r$, $1\leq j, r\leq n-1$, there exists $l_r\in I_r$ such that $X_{r,l_r}\cap \pi_k(X_{j,x_j})\neq \O $. In fact, for
each $r$, $\sum_{l\in I_r} P_{r,l}=1_{\CH}$. Hence $\psi(\tau_k)P_{j, x_j} \psi(\tau_k^{-1})=\sum_{l\in I_r} \psi(\tau_k)P_{j, x_j}
\psi(\tau_k^{-1})\wedge P_{r,l}$. As left side is a non-zero projection, there exists $l_r\in I_r$ such that $\psi(\tau_k)P_{j, x_j}
\psi(\tau_k^{-1})\wedge P_{r,l}$ is non-zero. Then, $X_{r,l_r}\cap \pi_k(X_{j,x_j})$ is a measurable set with positive measure.

This says that all the sets $X_{r,l_r}$, $1\leq r\leq n-1$, appear in the expression of $\pi_k(X_{j,x_j})$ as union of intersections of elements
of the $\sigma$-algebra of subsets of $X$. That is, if $$ \pi_k(X_{j,x_j})=\bigcup_{m^j\in L^j} X_{t_1^j,m_1^j}\cap X_{t_2^j,m_2^j}\cap \dots
\cap X_{t_{r^j}^j,m_{r^j}^j}\cap \dots $$ where $1\leq j \leq n-1$ and $m^j= (m^j_1, m^j_2,\dots, m_{r^j}^j, \dots)$, then for each $r$ there
exists $i$ such that $t_i^j=r$ and $m_i^j=l_r$.

Note that $l_r$ can depend on $j$, but we can choose it such that it does not depend on $j$. In fact, we want to see that for each $r$ there
exists $l_r\in I_r$ such that for all $j$, $1\leq j\leq n-1$, $X_{r,l_r}\cap \pi_k(X_{j,x_j})\neq {\O}$. Suppose that for all $l\in I_r$, there
exists $j'$ such that $X_{r,l}\cap \pi_k(X_{j',x_{j'}})={\O}$. Hence, $$ {\O}=\cup_{l\in I_r} (X_{r,l}\cap \pi_k(X_{j',x_{j'}}))=(\cup_{l\in I_r}
X_{r,l})\cap \pi_k(X_{j',x_{j'}})= X\cap \pi_k(X_{j',x_{j'}}) $$ which is a contradiction.

Compute $\cap_{j=1}^{n-1} \pi_k(X_{j,x_j})$, $$ \begin{aligned} &\bigcap_{j=1}^{n-1} \pi_k(X_{j,x_j})=\bigcap_{j=1}^{n-1} (\bigcup_{m^j\in L^j}
X_{t_1^j,m_1^j}\cap X_{t_2^j,m_2^j}\cap \dots \cap X_{t_{r^j}^j,m_{r^j}^j}\cap \dots)\\
                       &=\bigcup_{m^1\in L^1, \dots,m^{n-1}\in L^{n-1}}\left(X_{t_1^1,m_1^1}\cap
X_{t_2^1,m_2^1}\cap \dots \cap X_{t_{r^{1}}^1,m_{r^{1}}^1} \cap \dots\cap X_{t_1^k,m_1^k}\cap X_{t_2^k,m_2^k}\cap \right.\\
                        &\qquad\qquad\left.\cap\dots \cap X_{t_{r^k}^k,m_{r^k}^k} \cap\dots \cap
X_{t_1^{n-1},m_1^{n-1}}\cap X_{t_2^{n-1},m_2^{n-1}}\cap\dots \cap X_{t_{r^{n-1}}^{n-1},m_{r^{n-1}}^{n-1}}\cap \dots \right) \end{aligned} $$ But
if in a term of the union
 \begin{equation} \label{1}
\begin{aligned} X_{t_1^1,m_1^1}&\cap X_{t_2^1,m_2^1}\cap \dots \cap X_{t_{r^{1}}^1,m_{r^{1}}^1} \cap \dots\cap  X_{t_1^k,m_1^k}\cap
X_{t_2^k,m_2^k}\cap\\
                        &\cap\dots \cap X_{t_{r^k}^k,m_{r^k}^k} \cap\dots \cap
X_{t_1^{n-1},m_1^{n-1}}\cap X_{t_2^{n-1},m_2^{n-1}}\cap\dots \cap X_{t_{r^{n-1}}^{n-1},m_{r^{n-1}}^{n-1}}\cap \dots \end{aligned} \end{equation}
appears the subsets $X_{t_{\alpha}^l,m_{\alpha}^l}$ and $X_{t_{\alpha}^l,m_{\beta}^l}$, with the same first subindex $t_{\alpha}^l$, then the
second ones have to be equal, $\alpha=\beta$, or the intersection \ref{1} is empty. Therefore, each term of the union is intersection at most of
$n-1$ sets, or it is empty. But the $n-1$ sets $X_{r,l_r}$ are the unique sets which appear in the decomposition of $\pi_k(X_{j,x_j})$ for all
$j$. Then, they appear in each not empty term of the union.

Therefore $ \bigcap_{j=1}^{n-1} \pi_k(X_{j,x_j})=\cap_{r=1}^{n-1} X_{r,l_r}$ or it is empty. But we have already seen that it is non empty, so
$\pi_k(x)$ is well defined. \end{proof}

Under these conditions we can establish the following parametrization

\begin{thm}\label{teorema 2} Let $\psi:\B_n\rightarrow \CB(\CH)$ be a representation of the braid group $\mathcal{B}_n$ ($n\leq \infty$)
over a separable Hilbert space $\CH$, that satisfies the following conditions, for all $k,j$
\begin{enumerate}
\item $\psi_k \psi_k^*$ has discrete spectral decomposition such that $\psi_k \psi_k^*\neq \lambda 1_\CH$;
\item $\psi_k\psi_k^*\psi_j\psi_j^*=\psi_j\psi_j^*\psi_k\psi_k^*$;
\item if $N$ is the von Neumann algebra generated by $$\CF=\{\psi_k\psi_k^*: k=1, \dots, n-1\}$$
then  $\psi_k P \psi_k^{-1}\in N$, for all projection $P\in N$. That is $\B_n$ act in $N$ by conjugation.
\end{enumerate} Then there exist a $5$-tuple $(\pi, X, \mu, \nu, U)$ such that $\psi=\phi_{(\pi, X, \mu, \nu, U)}$, where
\begin{enumerate}
\item[(a)] $\pi$ is an action of the braid group on $X$;
\item[(b)] $(X, \mu)$ is a measure space with $\mu$ a $\pi$-quasi-invariant measure;
\item[(c)] $\nu: X\rightarrow \mathbb{N}\cup \{\infty\}$ is a $\pi$-invariant measurable function;
\item[(d)] $U:\B_n\times X\rightarrow \CB(\CH)$ is a cocycle on $\CH$, the direct integral Hilbert space associate to $(X, \mu, \nu)$,
such that
\begin{equation} \label{ecuacion (3)}
U(\tau_k,\cdot)U(\tau_k,\cdot)^*  {\text{    is a diagonalizable operator on   }} \CH,
\end{equation}
\end{enumerate} That is
 \begin{equation} \label{represent} (\psi_k f)(x)=(\psi(\tau_k)f)(x)= \sqrt{\frac{d\mu(\pi_k^{-1}x)}{d\mu(x)}}U(\tau_k, \pi_k^{-1}x)f(\pi_k^{-1}x) \end{equation}

 Conversely each $5$-tuple with those properties defines a representation of $\B_n$ on
 $\CH=\int_X^{\oplus} \CH_x d\mu(x)$ given by \ref{represent}
where the operators $\psi_k$ satisfy the conditions (1), (2), (3).
\end{thm}

\begin{rem} The conditions \ref{condicion (3)} and $(3)$ of the theorem are equivalent because $N$ is generated by the projections $P_{j,l}$.
\end{rem} \begin{proof} We have already seen that the von Neumann algebra $N$ generated by $\CF$ decomposes $\CH$ as the direct integral of the
Hilbert space $\CH_x$ of dimension $\nu(x)$ over $(X,\mu)$. Furthermore, by the lemma \ref{accion} there is an action $\pi$ of the braid group
$\B_n$ on $X$. We denote $\pi_k:=\pi(\tau_k)$ for each $k$, $1\leq k\leq n-1$.

As $N$ is generated by the commutative family $\widetilde{\CF}=\{P_{k,l}: l\in I_k, k=1, \dots, n-1 \}$ as von Neumann algebra, every element
$M_{\varphi}$ of $N$ is strong limit of linear combinations of the elements of ${\widetilde{\CF}}$. Then, by condition (3) we have the following
relations \begin{equation}\label{propiedad (1)}
 \psi_k M_{\varphi}= M_{\varphi \circ \pi_k^{-1}} \psi_k
\end{equation}

Now, we prove the $\pi$-invariance of the dimension function $\nu$. It is enough to see that $\dim \CH_x= \dim \CH_{\pi_k (x)}$ for almost all
$x$ and all $k$. We assume that this fails, then there is $E\subseteq X_m$ a set of positive measure such that $\pi_k(E) \subseteq X_{m-1}$. Let
$x\in E$ and let $\{h_1, \dots h_m\}$  be a base of the Hilbert space $\CH_m$. Consider the following vector functions $f_i(x)=\chi_F (x) h_i$
with $i=1, \dots, m$ and $F$ an arbitrary subset of $E$ of positive measure. These functions are linearly independent in $\CH$. Let $g_i(x)=
(\psi_k f_i)(x)$, as $\psi_k$ is invertible the set $\{ g_i\}$ is also linearly independent in $\CH$.

We want to see that for almost all $y\in \pi_k(F)$, the set $\{g_i(y)\}$ is linearly independent  in $\CH_y$. Assume that there is  a set $F$ of
positive measure such that this set is linearly dependent  for almost all $y\in \pi_k (F)$. Hence, suppose that $g_m(y)=\sum_{i=1}^{m-1} a_i
g_i(y)$. Then, by the linearity of $\psi_k$, $$ (\psi_k f_m)(y)= \sum_{i=1}^{m-1} a_i (\psi_k f_i) (y)= \left(\psi_k \left(\sum_{i=1}^{m-1}
a_if_i\right)\right)(y) $$

That is, $(\psi_k (f_m- \sum_{i=1}^{m-1} a_i f_i))(y)=0$ for almost all $y\in \pi_k(F)$. Therefore $f_m- \sum_{i=1}^{m-1} a_i f_i=0$ since
$\psi_k$ is invertible and $\mu(\pi_k(F))>0$. But this contradicts the
 linear independence of $\{f_i\}$.

Therefore $\{g_i(y)\}_{i=1}^m$ is  linearly independent in $\CH_y$. On the other hand, the dimension of $\CH_y$ is $m-1$ by the choose of $E$,
then there is a contradiction. Thus, $\nu (x)= \nu (\pi_k x)$ for almost all $x\in X$.

Now, we have to see that the measure $\mu $ is $\pi$-quasi-invariant, that is that $\mu (x)$ and $\mu (\pi_k x)$ have the same set of measure
zero. We are going to prove that $\mu(E)> 0$ if and only if $\mu(\pi_k E)> 0$ for all $k$. But $X$ is the disjoint union of the sets $X_m=\{
x\in X :\nu(x)=m\}$. Then, $E=E\cap X= \bigcup_{m\leq \infty} E\cap X_m$. Therefore $\mu(E)>0$ if and only if $\mu(E\cap X_m)>0$ for some $m$.
Hence, we may suppose that $E\subseteq X_m$.

Let $\chi_E $ be the characteristic function of the set $E$, let $h$ be a unitary vector in the space $\CH_m$, and consider the vector function
$x\rightarrow f(x)=\chi_E (x) h$. Then $$ \begin{array}{ll} <f,f>  &= \int_X <f(x), f(x)> d\mu (x) \\
       &= \int_X <\chi_E(x)h,\chi_E (x) h>d\mu (x)\\
       &= \int_E <h,h> d\mu(x) = \mu(E)
\end{array} $$ On the other hand, for each $k$, $1\leq k \leq n-1$ $$ \begin{array}{ll} \mu(E)=<f,f>  &= <\psi_k^{-1} \psi_k f,f>\\
              &= <\psi_k f, (\psi_k^{-1})^*f>\\
              &= <\psi_k M_{\chi_E}f,(\psi_k^{-1})^*M_{\chi_E}f> \\
              &= <M_{\chi_{\pi_k (E)}} \psi_k f, M_{\chi_{\pi_k (E)}} (\psi_k^{-1})^* f> \\
              &= \int_{\pi_k(E)} <(\psi_kf)(x),((\psi_k^{-1})^*f)(x)> d\mu (x)
\end{array} $$

Therefore, if we suppose that $\mu (E)>0$ and $\mu (\pi_k E)=0$ we obtain a contradiction. Conversely, let $E\subset X_m$. As $\nu$ is
$\pi$-invariant,  $\pi_k( E)\subset X_m$ too. We want to see that if $\mu(\pi_k E)>0$ then $\mu(E)>0$. Let $g$ be a vector function such that
$g(x)=\chi_{\pi_k (E)}(x) h$, where $h$ is a unitary vector of $\CH_m$. Hence $$<g,g> =\int_{\pi_k(E)} <h,h> d\mu(x) = \mu(\pi_k E) $$ But

$$ \begin{array}{ll} \mu(\pi_k E)=<g,g>  &= <\psi_k \psi_k^{-1} g,g>\\
              &= <\psi_k^{-1} g, \psi_k^* g>\\
              &= <\psi_k^{-1} M_{\chi_{\pi_k(E)}}g,\psi_k^*M_{\chi_{\pi_k(E)}}g> \\
              &= <M_{\chi_E} \psi_k^{-1} g, M_{\chi_E} \psi_k^* g> \\
              &= \int_{E} <(\psi_k^{-1}g)(x),(\psi_k^*g)(x)> d\mu (x)
\end{array} $$

Therefore, $\mu(\pi_k E)>0$ implies $\mu(E)>0$. Hence, the $\pi$-quasi-invariance of the measure is proved.

Consider the following operator on $\CH$, $(V_kf)(x)= \sqrt{\frac{d\mu (\pi_k x)}{d\mu (x)}} f(\pi_k x)$. Observe that $V_k$ is well defined by
the $\pi$-invariance of $\nu$ and is invertible, $(V_k^{-1}f)(x)= \sqrt{\frac{d\mu (\pi_k^{-1}x)}{d\mu (x)}}f(\pi_k^{-1}x)$. Furthermore, they
have the following property,

\begin{equation} \label{propiedad (2)}
 V_k M_{\varphi}=M_{\varphi \circ \pi_k} V_k
 \end{equation}

Let $U(\tau_k, \cdot):= V_k \psi_k$, it commutes with all the operators $M_{\varphi}$ of $N$ by \ref{propiedad (1)} and \ref{propiedad (2)}.
Then it is a decomposable operator and there exists $U(\tau_k, x):\CH_x\rightarrow \CH_x$ such that

$$ (U(\tau_k,\cdot) f)(x)=U(\tau_k, x) f(x)$$ for almost all $x$.

Then $\psi_k=V_k^{-1}U(\tau_k, \cdot)$, that is $$ \begin{aligned}
 (\psi_k f)(x)= \sqrt{\frac{d\mu (\pi_k^{-1} x)}{d\mu (x)}} U(\tau_k, \pi_k^{-1} x) f(\pi_k^{-1} x)
\end{aligned} $$

The relations of the braid group say that the operators $U(\tau_k, \cdot)$ verify the following equations, $$ \begin{aligned} U(\tau_k,
\pi_k^{-1}x) & U(\tau_{k+1}, \pi_{k+1}^{-1}\pi_k^{-1} x) U(\tau_k, \pi_k^{-1}\pi_{k+1}^{-1} \pi_k^{-1} x) = \\
   & = U(\tau_{k+1}, \pi_{k+1}^{-1}x) U(\tau_k,\pi_k^{-1}\pi_{k+1}^{-1}x)U(\tau_{k+1},\pi_{k+1}^{-1}\pi_k^{-1} \pi_{k+1}^{-1}x)
   \end{aligned}
$$ if $1\leq k\leq n-2$,

$$ U(\tau_k, \pi_{k}^{-1}x) U(\tau_j,\pi_{j}^{-1}\pi_{k}^{-1}x)= U(\tau_j, \pi_{j}^{-1}x) U(\tau_k, \pi_{k}^{-1}\pi_j^{-1}x) $$if $|j-k|>1$.

This means that $U:G\times X\rightarrow\CB(\CH)$ is a cocycle. Using the inner product of the direct integral space $\CH$, we obtain that
$(\psi_k^*f)(x)= \sqrt{\frac{d\mu(\pi_k x)}{d\mu(x)}} U(\tau_k, x)^* f(\pi_k x)$.

Then $$(\psi_k \psi_k^*f)(x)= \sqrt{\frac{d\mu(\pi_k^{-1} x)}{d\mu (x)}}\ \sqrt{\frac{d\mu (x)}{d\mu (\pi_k^{-1}x)}} U(\tau_k,
\pi_k^{-1}x)U(\tau_k, \pi_k^{-1}x)^* f(x) $$

But by condition (1), $$ (\psi_k \psi_k^*f)(x)=\sum_{l\in I_k} \lambda_{k,l} (P_{k,l}f)(x)= \sum_{l\in I_k} \lambda_{k,l} \chi_{X_{k,l}} (x) f(x)=
\lambda_{k,x_k}f(x) $$ then $U(\tau_k, \pi_k^{-1}x)U(\tau_k, \pi_k^{-1}x)^*=\lambda_{k,x_k}1_{\CH_x}$ and \ref{ecuacion (3)} is true.

 We have to prove the converse. By the theorem \ref{teorema 1}, $\psi$ given by \ref{represent} defines a representation of the braid
group $\B_n$ on the direct integral Hilbert space associated to the triple $(X, \mu, \nu)$. It rests to see that they satisfy the conditions
(1), (2) and (3).

$\psi_k \psi_k^*$ has discrete spectral decomposition because $U(\tau_k, \cdot)U(\tau_k, \cdot)^*$ is diagonalizable by \ref{ecuacion (3)}.

To see condition (2), we have that $$ \begin{array}{rl} (\psi_k\psi_k^* \psi_j \psi_j^*f)(x)= &\sqrt{\frac{d\mu (\pi_k^{-1}x)}{d\mu (x)}}\
\sqrt{\frac{d\mu (x)}{d\mu (\pi_k^{-1}x)}}U(\tau_k, \pi_k^{-1}x)U(\tau_k, \pi_k^{-1}x)^* \\
                          & \sqrt{\frac{d\mu
(\pi_j^{-1} x)}{d\mu (x)}} \ \ \sqrt{\frac{d\mu (x)}{d\mu (\pi_j^{-1}x)}} U(\tau_j, \pi_j^{-1}x) U(\tau_j, \pi_j^{-1}x)^* f(x) \end{array} $$

On the other hand, $$ \begin{array}{rl} (\psi_j\psi_j^* \psi_k \psi_k^*f)(x)=&\sqrt{\frac{d\mu (\pi_j^{-1} x)}{d\mu (x)}}\ \
\sqrt{\frac{d\mu(x)}{d\mu (\pi_j^{-1}x)}}U(\tau_j, \pi_j^{-1}x)U(\tau_j, \pi_j^{-1}x)^* \\
                         &\sqrt{\frac{d\mu (\pi_k^{-1} x)}{d\mu (x)}}\ \  \sqrt{\frac{d\mu (x)}
                         {d\mu (\pi_k^{-1}x)}} U(\tau_k, \pi_k^{-1}x)U(\tau_k, \pi_k^{-1}x)^* f(x)
\end{array} $$

But, by \ref{ecuacion (3)}, $U(\tau_k, \pi_k^{-1}x)U(\tau_k, \pi_k^{-1}x)^*f(x)=\lambda_{k}(x) 1_{\CH_x} f(x)$. In the same way $U(\tau_j,
\pi_j^{-1}x)U(\tau_j, \pi_j^{-1}x)^*(x)f(x)= \lambda_{j}(x)1_{\CH_x} f(x)$. Therefore, $U(\tau_k, \pi_k^{-1}x)U(\tau_k, \pi_k^{-1}x)^*$ commutes
with $U(\tau_j, \pi_j^{-1}x)U(\tau_j, \pi_j^{-1}x)^*$. Thus, \ref{condicion (2)} is proved.

For the last condition, we have that $$ \begin{array}{rl} (\psi_k P_{j,l}\psi_k^{-1}f)(x)&= \sqrt{\frac{d\mu (\pi_k^{-1} x)}{d\mu (x)}}U(\tau_k,
\pi_k^{-1}x) \chi_{X_{j,l}} (\pi_k^{-1}x)\\
                         &\qquad \qquad\qquad\qquad\sqrt{\frac{d\mu (x)}{d\mu (\pi_k^{-1} x)}}U(\tau_k, \pi_k^{-1}x)^{-1}f(\pi_k \pi_k^{-1}
                         (x))\\
&=\chi_{X_{j,l}}(\pi_k^{-1}(x))f(x)=(P_{\pi_k(X_{j,l})}f)(x) \end{array} $$ hence, (3) is true as we wanted. \end{proof}

\begin{rem} Observe that the proof of the Theorem \ref{teorema 2} does not use special properties of the braid group $\B_n$, then we can
substitute $\B_n$ for any group given by generators and relations. The proof in general only differs in the relations of the cocycle $U$ that
must depend on the relations of the group.

A similar results was obtained by Varadarajan for unitary representations of a group $G$ (see \cite{V} Theorem 6.11, pag 220), but his data is a
systems of imprimitivity. However, we propose the family of projections $\widetilde{\CF}$ such that $(\psi,\widetilde{\CF})$ can be a systems of
imprimitivity. \end{rem}

\begin{cor} \label{subespacio} Let $(\pi, X, \mu, \nu, U)$ be a $5$-tuple which satisfies the conditions of the Theorem \ref{teorema 2}. Let
$\phi$ be the representation of $\B_n$, $n\leq \infty$, associated to that $5$-tuple such that $\phi(\tau_k)$ is a self-adjoint operator for all
$k$, $1\leq k\leq n-1$. Then, every invariant closed subspace $\CK$ of the direct integral Hilbert space $\CH$ associated to $(X, \mu, \nu)$, is
a direct integral Hilbert space associated to $(X_{\CK}, \mu_{\CK}, \nu_{\CK})$, where $X_{\CK}\subset X$, $\mu_{\CK}=\mu_{|_{X_\CK}}$ and
$\nu_{\CK}(x)\leq \nu(x)$ for all $x\in X_{\CK}$. \end{cor} \begin{proof} Let $\mathcal{K}\subseteq \CH$ be a closed invariant subspace, then
$\CK$ is also invariant for $\phi(\tau_k)^*=\phi(\tau_k)$. Thus, $\phi(\tau_k)$ commutes with $P_\CK$, the orthogonal projection onto the
subspace $\CK$.

Therefore, $\widetilde{\phi_k}:=P_\CK \phi_k P_\CK$ is a representation of $\B_n$ on $\CK$ that verifies the hypothesis of the Theorem
\ref{teorema 2}. In fact, $$ \widetilde{\phi_k} \widetilde{\phi_k}^*=P_\CK \phi_k P_\CK \phi_k^*P_\CK=P_\CK \phi_k \phi_k^* P_\CK= \sum_{l\in I_k}
\lambda_{k,l} P_\CK P_{k,l} P_\CK=\sum_{l\in I_k} \lambda_{k,l} Q_{k,l} $$ and $$ \begin{aligned} \widetilde{\phi_k}
\widetilde{\phi_k}^*\widetilde{\phi_j} \widetilde{\phi_j}^* &=P_\CK \phi_k P_\CK \phi_k^*P_\CK P_\CK \phi_j P_\CK \phi_j^*P_\CK =P_\CK
\phi_k\phi_k^*\phi_j\phi_j^*P_\CK\\ &=P_\CK \phi_j\phi_j^*\phi_k\phi_k^*P_\CK= \widetilde{\phi_j} \widetilde{\phi_j}^*\widetilde{\phi_k}
\widetilde{\phi_k}^* \end{aligned} $$ Finally $$ \widetilde{\phi_k} Q_{j,l} \widetilde{\phi_k}^{-1}=P_\CK \phi_k P_\CK P_\CK P_{k,l} P_\CK P_\CK
\phi_k^{-1} P_\CK= P_\CK \phi_k P_{k,l}\phi_k^{-1} P_\CK \in \widetilde{N}:= P_\CK N P_\CK $$ Then, $\widetilde{\phi}=\phi_{(\pi_\CK, X_\CK,
\mu_\CK, \nu_\CK, P_\CK U P_\CK)}$.

Moreover $$ X_\CK=\{ (x_1, \dots, x_{n-1})\in X :  P_\CK P_{k, x_k} P_\CK \neq 0 \text{    for all   } k, 1\leq k \leq n-1\}\subseteq X $$ and
$$ \CK_x=\overline{\cap_k \IM (P_\CK P_{k,x_k} P_\CK)}=\overline{\IM P_\CK \cap(\cap_k \IM P_{k, x_k})}=\CK \cap \CH_x $$ Therefore,
$\nu_{\CK}\leq \nu$ and $\mu_{\CK}=\mu_{|_{X_\CK}}$. \end{proof}

\section{Well known examples}\label{Ejemplos}

\subsection{The standard representation} \label{Standart}

As long as we know, the first time that the standard representation appeared was in a work of Dian-Min Tong, Shan-De Yang and Zhong-Qi Ma [TYQ].
Sysoeva probed that it is the unique irreducible  representation  of $\B_n$ of dimension $n$, if $n\geq 9$, [S]. It is defined in the following
way for each $k=1,\dots ,n-1$,

$$ \rho(\tau_k)= \left(\begin{array}{ccccccc}
          1      &\   &\      &\   &\   &\        &\  \\
          \      &1   &\      &\   &\   &\        &\  \\
          \      &\   &\ddots &\   &\   &\        &\  \\
          \      &\   &\      &0   &t   &\        &\  \\
          \      &\   &\      &1   &0   &\        &\  \\
          \      &\   &\      &\   &\   &\ddots   &\  \\
          \      &\   &\      &\   &\   &\        &1
\end{array}\right) $$ where $t$ is in the place $(k, k+1)$ and $\rho_k:=\rho(\tau_k)$ acts on $\CH= \mathbb{C}^n$.

They verify the conditions of the theorem. In fact, $$ \rho_k \rho_k^*= \left(\begin{array}{ccccccc}
          1      &\   &\      &\       &\   &\        &\  \\
          \      &1   &\      &\       &\   &\        &\  \\
          \      &\   &\ddots &\       &\   &\        &\  \\
          \      &\   &\      &|t|^2   &\   &\        &\  \\
          \      &\   &\      &\       &1   &\        &\  \\
          \      &\   &\      &\       &\   &\ddots   &\  \\
          \      &\   &\      &\       &\   &\        &1
\end{array}\right) $$
 is a linear operator with two eigenvalue $\lambda_{k,0}=1$ and
$\lambda_{k,1}=|t|^2$; the projections associated to them are the orthogonal projections over the spaces $W_{k,0}=\Ker (\rho_k \rho_k^*- 1_\CH)$
and $W_{k,1}=\Ker (\rho_k \rho_k^* -|t|^2 1_\CH)$ respectively.

As $\rho_k \rho_k^*$ are diagonal for all $k$, they commute. It is easy to see that for every $l, k$, $0\leq l\leq 1$ and $1\leq k\leq n-1$ we
have that $\rho_k P_{j,l}\rho_k^{-1}=P_{j,l}$ if $j\neq k, k+1$. While $\rho_k P_{k,l}\rho_k^{-1}=P_{k+1, l}$ and $\rho_k P_{k+1,
l}\rho_k^{-1}=P_{k,l}$ if $k\leq n-2$. Finally $\rho_{n-1}P_{n-1, 0}\rho_{n-1}^{-1}=\wedge_{j=1}^{n-1} P_{j, 1}$ and $\rho_{n-1}P_{n-1,
1}\rho_{n-1}^{-1}=\sum_{j=1}^{n-1} P_{j, 0}$. Hence, the condition (3) of theorem \ref{teorema 2} is verified. Following the proof of theorem
\ref{teorema 2}, we can see that the standard representation is equivalent to the representation $\phi_{(\pi, X, \mu, \nu , U)}$ given by the
$5$-tuple $(\pi, X, \mu, \nu ,U)$, where \begin{enumerate} \item $X=\{(x_1, \dots , x_{n-1}) : x_i=0,1;\    i=1,\dots, n-1\}$, \item If
$\delta_k$ is the $(n-1)$-tuple with $1$ in the place $k$ and zero elsewhere for $k=1, \dots, n-1$, and $\delta_0=(0, \dots, 0)$ $$\mu (x)=\left
\{\begin{array}{ll}
    1       &\text{   if  } x=\delta_k \textrm{  for some  } k=0, \dots, n-1   \\
    0       &\text{  in other case  }
\end{array}\right. $$ \item The action $\pi$ is defined for almost all $x\in X$ by $$\pi_k(x_1, \dots, x_{n-1})=(x_1, \dots, x_{k-1}, x_{k+1},
x_k, \dots, x_{n-1})$$ if $k=1, \dots, n-2$. And $\pi_{n-1}$ is defined by $$ \begin{array}{rl} \pi_{n-1}(x_1, \dots, x_{n-1})= \left
\{\begin{array}{ll}
    (0,\dots, 0)            & \text{  if  } (x_1, \dots, x_{n-1})=(0, \dots, 1)   \\
    (0, \dots, 1)            &\text{  if  } (x_1, \dots, x_{n-1})=(0, \dots, 0)     \\
    (x_1, \dots, x_{n-1})    & \text{ in other cases   }
\end{array}\right. \end{array} $$

\item $\nu(x)=1$ for almost all $x\in X$, then $\CH_x=\C$, \item  $U(\tau_k, \tau_k^{-1}x)= 1+(t-1)x_{k+1} \in \C$. \end{enumerate}

Note that the measure $\mu$ is invariant by the action $\pi$. Therefore $p_k(x)=\frac{d\mu(\pi_k(x)}{d\mu(x)}=1$ for almost all $x$.

If $\{\beta_j: j=1, \dots, n\}$ is the canonical basis of $\C^n$, and $f_j=\chi_{\delta_j}$ is the characteristic function of the element
$\delta_j$ define $$ \begin{array}{rl}
 \alpha :\C^n & \rightarrow \CH   \\
      \beta_j &\mapsto \left
\{\begin{array}{ll}
    f_j   &\text{   if    } j=1, \dots, n-1   \\
    f_0   &\text{  if    } j=n
\end{array}\right. \end{array} $$

Then $\alpha (\rho_k (\beta_j))=\phi_k(\alpha(\beta_j))$ for all $j$,$1\leq j\leq n$. Thus, the representations are equivalent.

Observe that $|X|=2^n$, however the $\dim \CH=n$. This happens because there are only $n$ points of $X$ with non zero  measure.

Note that when $n=\infty$, the action of the element $\pi_k$ consists on the permutation of the places $k$ and $k+1$. Hence, we obtain a
representation of $\B_{\infty}$ as a generalization of the standard representation.

\subsection{Local Representations}

Let $V$ be a vector space and $c: V\otimes V\rightarrow V\otimes V$ an invertible lineal operator, $c$ satisfies the braid equation if it
satisfies the following equality in $V \otimes V \otimes V= V^{\otimes 3}$ $$ (c\otimes 1_V)(1_V\otimes c)(c\otimes 1_V)=(1_V\otimes c)(c
\otimes 1_V)(1_V\otimes c) $$

In this case $(V,c)$ is called a braid vector space and $c$ a $R$-matrix. On the space $V^{\otimes n}$ we may define a representation of $\B_n$
in the following way. For each $k=1, \dots , n-1$ let $c(\tau_k):=c_k: V^{\otimes n}\rightarrow V^{\otimes n}$ given by $c_k=1_{k-1}\otimes c
\otimes 1_{n-k-1}$, where $1_k= 1_V\otimes \dots \otimes 1_V$, is the $k$ times tensor product of the identity on $V$; this representation is
called {\it{local}}.


If there exists a basis $\beta=\{v_1, \dots, v_m\}$ of $V$ such that $c(v_i\otimes v_j)= q_{ij} v_j\otimes v_i$ for all $i,j$, $1\leq i,j\leq m$
where the $q_{ij}$ are non zero complex numbers, the representation is called of {\it{diagonal}} type \cite{AG}. In this case, the operators
$c_k$ satisfy the hypothesis of the theorem \ref{teorema 2}. In fact, $$ c_k(v_{j_1}\otimes \dots \otimes v_{j_n})= q_{j_k, j_{k+1}}
v_{j_1}\otimes\dots \otimes v_{j_{k+1}}\otimes v_{j_k} \otimes \dots \otimes v_{j_n} $$ where $j_i \in \{1, \dots , \dim V\}$ for all $i=1,
\dots, n$. Then

$$ c_k^*(v_{j_1}\otimes \dots \otimes v_{j_n})=\overline{q_{j_{k+1}, j_k}} v_{j_1}\otimes\dots \otimes v_{j_{k+1}}\otimes v_{j_k} \otimes \dots \otimes v_{j_n} $$ Then, $$ c_k c_k^* (v_{j_1}\otimes \dots \otimes v_{j_n})=|q_{j_k, j_{k+1}}|^2 v_{j_1}\otimes\dots \otimes v_{j_k}\otimes v_{j_{k+1}} \otimes \dots \otimes v_{j_n} $$ If $S=\{(a,b): a,b\in \{1, \dots, \dim V\} \}$, we have that $$ c_kc_k^* =\sum_{(a,b)\in S} |q_{(a,b)}|^2 P_{k,(a,b)} $$ where $P_{k, (a,b)}$ is the projection over the subspace of $V^{\otimes n}$ generated by all the vectors $v_{j_1}\otimes \dots \otimes v_{j_n}$ such that $j_k=a$ y $j_{k+1}=b$. 

As $c_kc_k^*$ is diagonal for all $k$, $c_kc_k^*$ and $c_jc_j^*$ commute for all $j$.

We have to see the condition (3) of theorem \ref{teorema 2}. In fact, $$ c_kP_{j, (a, b)}c_k^{-1}=\left\{ \begin{array}{ll}
              P_{j, (a,b)}        &  \text{    if    } |j-k|>1 \\
              P_{j, (b, a)}       & \text{    if    }  j=k       \\
              \sum_{c=1}^{\dim V} P_{k-1, (a, c)}\wedge P_{k, (c, b)}       & \text{    if    }  j=k-1  \\
              \sum_{c=1}^{\dim V} P_{k, (a, c)}\wedge P_{k+1, (c, b)}       & \text{    if    }  j=k+1
\end{array}\right. $$

Therefore, following the proof of Theorem \ref{teorema 2}, the representation is equivalent to $\phi_{(\pi, X, \mu, \nu,  U)}$, where
\begin{enumerate} \item $X=\{(x_1, \dots , x_{n-1}) : x_i=(a_i, b_i) , a_i,b_i \in \{1, \dots,
 \dim V\}\}$,
\item  $\mu ((a_1, b_1), \dots, (a_{n-1}, b_{n-1}))=\left \{\begin{array}{ll}
    1       &\text{   if  } b_i=a_{i+1} \text{    for all   } i=1, \dots, n-2   \\
    0       &\text{  in other case  }
\end{array}\right.$ \item The action $\pi$ is defined for almost all $x\in X$ by $$ \begin{aligned} \pi_k&((a_1, a_2),(a_2, a_3), \dots,
(a_{n-1}, b_{n-1}))=\\ &\qquad=((a_1, a_2),\dots, (a_{k-1}, a_{k+1}), (a_{k+1}, a_k),(a_k, a_{k+2}), \dots, (a_{n-1}, b_{n-1})),
\end{aligned} $$ and $\pi_k^{-1}(x)=\pi_k(x)$ for almost all $x$. \item  $\nu(x)=1$ for all $x\in X$, \item  $U(\tau_k,
\tau_k^{-1}x)=q_{x_k}$. \end{enumerate}

Let $\alpha :V^{\otimes n} \rightarrow \CH$ be the lineal operator defined in the basis of $V^{\otimes n}$ by $$ \alpha(v_{j_1}\otimes \dots
\otimes v_{j_n})=\chi_{((j_1, j_2), (j_2,j_3), \dots ,(j_{n-1}, j_n))} $$

It verifies that $\alpha(c_k(v_{j_1}\otimes \dots \otimes v_{j_n}))= \phi_k (\alpha(v_{j_1}\otimes \dots \otimes v_{j_n}))$ showing the
equivalence of the representations.

In the same way as in the previous example, if $n=\infty$ this construction permits us to generalize
 this representation to one of $\B_{\infty}$ in a natural way.

\


Other local representation is the following (\cite{AS}). Let $H$ be any group and let $T\subset H$ be such that for all $g\in H$ and $t\in T$,
$gtg^{-1}\in T$, that is $T$ is closed by conjugation of elements of $H$. Let $\gamma : H \times T\rightarrow \mathbb{C}-\{0\}$ be a function
such that for all $g,h\in H$, and $t\in T$

\begin{enumerate} \item $\gamma(1,t)=1$ \item $\gamma(gh,t)=\gamma(g,hth^{-1}) \gamma(h,t)$ \end{enumerate}

Let $V$ be the complex vector space with orthonormal basis $\{v_t: t\in T\}$ and let $c$ be defined by $c(v_s\otimes v_t)=
\gamma(s,t)v_{sts^{-1}}\otimes v_s$. The conditions imposed to $\gamma$ ensure that $c$ verify the braid equation.

Then $$c_k(v_{g_1}\otimes \dots \otimes v_{g_n})= \gamma(g_k,g_{k+1}) v_{g_1}\otimes \dots v_{g_{k-1}}\otimes v_{g_k g_{k+1} g_k^{-1}} \otimes v_{g_k} \otimes v_{g_{k+2}}\otimes \dots \otimes v_{g_n} $$ Taking  the inner product in $V$, $$ c_k^*(v_{g_1}\otimes \dots \otimes v_{g_n})= \gamma(g_{k+1}, g_{k+1}^{-1}g_k g_{k+1}) v_{g_1}\otimes \dots \otimes v_{g_{k-1}}\otimes v_{g_{k+1}} \otimes v_{g_{k+1}^{-1} g_k g_{k+1}} \otimes \dots \otimes v_{g_n} $$ We obtain that $$ c_kc_k^*(v_{g_1}\otimes \dots \otimes v_{g_n})=|\gamma(g_{k+1}, g_{k+1}^{-1}g_k g_{k+1})|^2 v_{g_1}\otimes \dots \otimes v_{g_n} $$ If $S=\{(a,b): a,b\in  T \}$ we have that $$ c_kc_k^* =\sum_{(a,b)\in S} |\gamma(b, b^{-1} a b)|^2 P_{k,(a,b)} $$ where $P_{k, (a,b)}$ is the projection over the subspace of $V^{\otimes n}$ generated for all the vectors $v_{g_1}\otimes \dots \otimes v_{g_n}$ such that $g_k=a$ and $g_{k+1}=b$. 

As $c_kc_k^*$ is diagonal for all $k$, it commutes with $c_jc_j^*$ for all $j$. We have just checked only the conditions (1) and (2) of Theorem
\ref{teorema 2}. For the condition (3) we have that, $$ c_kP_{j, (a, b)}c_k^{-1}=\left\{ \begin{array}{ll}
              P_{j, (a,b)}        &  \text{    if    } |j-k|>1 \\
              P_{j, (aba^{-1}, a)}       & \text{    if    }  j=k       \\
              \sum_{c\in T} P_{k-1, (a, c)}\wedge P_{k, (c, b)}       & \text{    if    }  j=k-1  \\
              \sum_{c\in T} P_{k, (cac^{-1}, c)}\wedge P_{k+1, (c, b)}       & \text{    if    }  j=k+1
\end{array}\right. $$

Then this representation is equivalent to
 $\phi:=\phi_{(\pi, X, \mu, \nu, U)}$, where
\begin{enumerate} \item $X=\{(x_1, \dots , x_{n-1}) : x_i=(a_i, b_i), a_i, b_i\in T\}$, \item  $\mu ((a_1, b_1), \dots, (a_{n-1},
b_{n-1}))=\left \{\begin{array}{ll}
    1       &\text{   if  } b_i=a_{i+1} \text{    for all   } i=1, \dots, n-2   \\
    0       &\text{  in other cases  }
\end{array}\right.$ \item The action $\pi$ is defined for almost all $x\in X$ by $$ \begin{aligned} \pi_k&((a_1, a_2),(a_2, a_3) \dots,
(a_{n-1}, b_{n-1}))= \\ &=((a_1, a_2),\dots, (a_{k-1}, a_k a_{k+1} a_k^{-1}), (a_k a_{k+1} a_k^{-1},a_k),
               (a_k, a_{k+2}), \dots, (a_{n-1}, b_{n-1})),
\end{aligned} $$ and $$ \begin{aligned} \pi_k^{-1}&((a_1, a_2),(a_2,a_3) \dots, (a_{n-1}, b_{n-1}))=((a_1, a_2),\dots,(a_{k-1}, a_{k+1}), \\
& \qquad\qquad\qquad\qquad(a_{k+1}, a_{k+1}^{-1}a_k a_{k+1}),(a_{k+1}^{-1}a_k a_{k+1}, a_{k+2}), \dots, (a_{n-1}, b_{n-1})), \end{aligned}
$$
 Note that $\pi_k^{-1}\neq \pi_k$

\item  $\nu(x)=1$ for all $x\in X$, \item  $U(\tau_k, \tau_k^{-1}x)=\gamma(x_k)=\gamma(a_k, b_k)$.

\end{enumerate}

The map $\alpha :V^{\otimes n} \rightarrow \CH$ defined in the basis of $V^{\otimes n}$ by $$ \alpha (v_{g_1}\otimes \dots \otimes
v_{g_n})=\chi_{((g_1,g_2), (g_2, g_3), \dots ,(g_{n-1}, g_n))} $$ gives the equivalence between the representations.

Likewise in the previous examples, this representations can be generalized to $\B_{\infty}$ in a natural way taking $n=\infty$.

\subsection{The Burau Representation}

The Burau representation was defined by Burau in $1936$. Formanek [F] proved that this representation and  an ``extended" one of it, are the
unique irreducible representations of $\B_n$ of dimension $n-1$, if $n\geq 7$. The Burau representation is defined by the following operators
acting on $\C^{n-1}$, $$ \rho_1= \left(\begin{array}{cccccccc}
           -t    &1    &\    &\       &\   &\    &\      &\  \\
            0    &1    &\    &\       &\   &\    &\      &\  \\
            \    &\    &1    &\       &\   &\    &\      &\  \\
            \    &\    &\    &\ddots  &\   &\    &\      &\  \\
            \    &\    &\    &\       &1   &\    &\      &\  \\
            \    &\    &\    &\       &\   &1    &\      &\  \\
            \    &\    &\    &\       &\   &\    &\ddots &\  \\
            \    &\    &\    &\       &\   &\    &\      &1
\end{array}\right), $$ $$ \rho_k= \left(\begin{array}{ccccccccc}
            1    &\    &\    &\       &\   &\    &\      &\       &\
  \\
            \    &1    &\    &\       &\   &\    &\      &\       &\
  \\
            \    &\    &1    &\       &\   &\    &\      &\       &\
  \\
            \    &\    &\    &\ddots  &\   &\    &\      &\       &\
  \\
            \    &\    &\    &\       &1   &0    &0      &\       &\
  \\
            \    &\    &\    &\       &t   &-t   &1      &\       &\
  \\
            \    &\    &\    &\       &0   &0    &1      &\       &\
  \\
            \    &\    &\    &\       &\   &\    &\      &\ddots  &\
  \\
            \    &\    &\    &\       &\   &\    &\      &\       &1
\end{array}\right),   \ \ \ \ if \ \ 2\leq k\leq n-2, $$ where $-t$ is in the place $(k,k)$, $$ \rho_{n-1}= \left(\begin{array}{ccccccccc}
            1    &\    &\    &\       &\   &\    &\      &\   &\  \\
            \    &1    &\    &\       &\   &\    &\      &\   &\  \\
            \    &\    &1    &\       &\   &\    &\      &\   &\  \\
            \    &\    &\    &\ddots  &\   &\    &\      &\   &\  \\
            \    &\    &\    &\       &1   &\    &\      &\   &\  \\
            \    &\    &\    &\       &\   &1    &\      &\   &\  \\
            \    &\    &\    &\       &\   &\    &\ddots &\   &\  \\
            \    &\    &\    &\       &\   &\    &\      &1   &0  \\
            \    &\    &\    &\       &\   &\    &\      &t   &-t
\end{array}\right), $$ These operators do not satisfy the condition that $\rho_k \rho_k^*$ commute with $\rho_j \rho_j^*$ for all $j$, therefore
they do not correspond to a representation given by a $5$-tuple as in Theorem \ref{teorema 2}.

\subsection{Other examples}\label{EG} Consider the representations given in \cite{EG}, they are constructed in the following way. Choose $n$ non
negative integers $z_1, z_2, \dots, z_n$, no necessarily different. Let $X$ be the set of all the possible $n$-tuples obtained by permutation of
the coordinates of the fixed $n$-tuple $(z_1, \dots, z_n)$. Let $V$ be a complex vector space with orthonormal basis $\beta=\{v_x : x\in X\}$.
Then the dimension of $V$ is the cardinality of $X$.

Define $\phi:\B_n \rightarrow \Aut (V)$, such that $$ \phi(\tau_k)(v_x)= q_{x_k, x_{k+1}} v_{\sigma_k(x)} $$ where $q_{x_k, x_{k+1}}$ is a
non-zero complex number that depends on $x$. Really it only depends on the places $k$ and $k+1$ of $x$. Define $$ \sigma_k(x_1, \dots,
x_n)=(x_1, \dots, x_{k+1}, x_k, \dots, x_n)) $$

It is easy to see that this representation is given by the $5$-tuple $(\pi, X', \mu, \nu, U)$, where \begin{enumerate} \item Let $Y:=\{z_1, z_2,
\dots, z_n \}$, then $X':= Y^n$; \item $\pi(\tau_k)(x_1, \dots, x_k, x_{k+1}, \dots, x_n)=(x_1, \dots, x_{k+1}, x_k \dots, x_n)$, $1\leq k\leq
n-1$; \item Let $a:=(z_1, z_2, \dots, z_n)\in X'$. $\mu$ is defined by $$\mu(x)=1   \Leftrightarrow {\text{  there exists   }} \tau\in \B_n
{\text{  such that   }} x=\pi(\tau)a ; $$ \item $\nu(x)=1$, for all $x\in X'$; \item $U(\tau_k, \tau_k x):= q_{x_k, x_{k+1}}$. \end{enumerate}

In \cite{EG} the reader can find a necessary condition for this representation to be irreducible and the following explicit family of
irreducible representations. Let $z_1=\dots=z_m=1$, $z_{m+1}=\dots =z_n=0$ and $$ q_{x_k, x_{k+1}}=\left \{\begin{array}{ll}
    1       &\text{    if   } x_k=x_{k+1}   \\
    t       &\text{    if   } x_k\neq x_{k+1}
\end{array}\right. $$ where $t$ is a real number, $t\neq 0, 1, -1$. The dimension of this representation is $\left( \begin{smallmatrix}
n\\\noalign{\medskip}m \end{smallmatrix}\right)=\frac{n!}{m! (n-m)!}$. If $m=1$, it is equivalent to the standard representation.

\section{Quasi-invariant measure}\label{medidas}

In this section we analyze the $\pi$-quasi-invariance condition of the measure. We will give examples of measures that satisfy this condition.

\subsection{Discrete Measures}

A measure $\mu$ over a set $X$ is called discrete if the $\sigma$-algebra of measurable sets is $\mathcal{P}(X)$ the set of all the subsets of
$X$. In particular, single points are measurable sets.

Let $X=\{(x_1, \dots, x_{n-1}): x_i\in I_i\}$, $n\leq \infty$, for some index sets $I_i$ and let $a$ be a fixed element of $X$. Then an easy example
of $\pi$-quasi-invariant discrete measure is the following $$ \mu_a(x)=1 \Longleftrightarrow {\text{  there exists   }} \tau\in \B_n {\text{
such that   }} x=\pi(\tau)a $$ That is, $\mu_a$ is concentrated in $\pi(\B_n) a$, the orbit of $a$. This is the case in examples \ref{Standart}
and \ref{EG}.

More generally, let $Y$ be a subset of $X$, $\mu=\sum_{a\in Y} \mu_a$ is a $\pi$-quasi-invariant discrete measure. It is concentrated in 
$\cup_{a\in Y} \pi(\B_n) a$, the union of the orbits $\pi(\B_n) a$, with $a\in Y$.

\subsection{Product Measures}

Consider that the space $X$ is a direct product of $n-1$ set $Y_k$, $n\leq\infty$. If we assign to each set $Y_k$ a measure $\mu_k$, we can
construct the product measure for $X$.

Remember the construction of the product space. Assume that for all $k$, $(Y_k, \mu_k)$ is a measure space and that $\mu_k(Y_k)=1$. If $X$ is a
product of a finite number of $Y_k$, this last condition is not necessary.

We assign to $X$ the $\sigma$-algebra $R$ of Borel subsets, generated by ``cylindrical" sets. These sets are defined by $$ X(a_{i_1}, \dots,
a_{i_r})=\{(x_1, \dots, x_{n-1}): x_{i_j}=a_{i_j}, j=1, \dots, r\} $$ where $a_{i_j}\in Y_{i_j}$ for all $j=1, \dots, r$, $r<\infty$.

Now, we have to define a measure over this $\sigma$-algebra $R$. It is enough to define the measure of the cylindrical set (see [H] pags
155-160), that is $$ \mu(X(a_{i_1}, \dots, a_{i_r}))= \prod_{j=1}^r \mu_{i_j}(a_{i_j})>0 $$

Let $\pi$ be an action of $\B_n$ on $X$. We analyze the $\pi$-quasi-invariance of product measures. In order to that we suppose that all the
$Y_k$ are equal, that is $Y_k=Y$ for all $k$. Furthermore, we assume that $\pi_k(X(a_{i_1}, \dots, a_{i_r}))=X(b_{j_1}, \dots, b_{j_s})$ then,
$$ \begin{aligned} &\mu(\pi_k(X(a_{i_1}, \dots, a_{i_r})))=\mu(X(b_{j_1}, \dots, b_{j_r}))=\prod_{l=1}^r \mu_{j_l}(b_{j_l})=\\ &\qquad
\qquad\qquad \qquad \qquad \qquad = \frac{\prod_{l=1}^s \mu_{j_l}(b_{j_l})}{\prod_{l=1}^r \mu_{i_j}(a_{i_j})} \mu(X(a_{i_1}, \dots, a_{i_r}))
\end{aligned} $$ Hence, if $A$ is a cylindrical set there exists a complex number $c_{k, A}$ such that $$ \mu(\pi_k(A))=c_{k,A} \mu(A) $$

\begin{prop}\label{medida producto} Let $(X=\prod_{k=1}^{n-1} Y, \mu=\prod_{k=1}^{n-1}\mu_k)$ be a product measure space with $n\leq \infty$. If
$Y=\{a^1, \dots, a^t\}$ is a finite set and each $\pi_k$ acts changing only finite coordinates $x_j$ of $x$, then $\mu$ is
$\pi$-quasi-invariant. \end{prop}

\begin{proof} We have to prove that $E$ is a set with $\mu(E)=0$ if and only if $\mu(\pi_k(E))=0$ for all  $k=1, \dots, n-1$.

Let $E$ be a set of measure zero. By definition of the $\sigma$-algebra $R$, for each $\varepsilon> 0$, there is a family of disjoint
cylindrical sets $\{E_l\}$ such that $E\subset \bigcup_{l\in \N} E_l$ and $\mu(\bigcup E_l)< \varepsilon $.

Given $k=1, \dots, n-1$, we want to see that $\mu(\pi_k(E))=0$. It is enough to prove that $\mu(\pi_k(\bigcup E_l))<\delta(\varepsilon) $.
Assume that $\pi_k$ changes the places $i_1^k, \dots, i_{r_k}^k$. For each $s\in \{1, \dots, r_k\}$ let $J=J(s)=(j_1, \dots, j_s)$ and
$a=a(s)=(a_1, \dots, a_s)$ be a multi-index such that $j_m \in \{i_1^k, \dots, i_{r_k}^k\}$ and $a^m\in Y$ for all $m$, $1\leq m\leq s$. Let
$F^{J}_{a}$  be the union of all the cylindrical set $E_l$ that fix the places $j_m$ by the value $a^{y_m}$ and let $F$ be  the union of the
remaining cylindrical sets. Then, $$ \begin{aligned} &\mu\left(\pi_k(\bigcup E_l)\right)= \mu \left(\pi_k (F \cup \bigcup_{s}\bigcup_{J,a} F^J_a
)\right)=\mu \left(\pi_k(F) \cup \bigcup_{s} \bigcup_{J,a}\pi_k (F^J_a)\right)= \\ &\qquad \qquad \leq \mu(\pi_k(F)) +\sum_{s} \sum_{J, a}
\mu(\pi_k(F^J_a)) =\mu(F) + \sum_{s} \sum_{ J, a} c_{k,J, a} \mu(F^J_a)\leq \\ & \qquad \qquad  \leq (1+ \sum_{s}\sum_{ J, a}  c_{k,J, a})
\mu\left(\bigcup E_l\right) \end{aligned} $$ As $\sum_{s} \sum_{ J, a} c_{k, J, a}$ is finite, this last expression is arbitrarily small.

With a similar argument, changing the role of $\pi_k$ by $\pi_k^{-1}$, we can shows that if $E$ is a measurable set such that $\mu(E)=0$ then
$\mu(\pi_k^{-1}E)=0$. \end{proof}

\begin{exa} We consider the Lebesgue measure on the interval $[0, 1]$. Let $n=\infty$ and $X$ be the set of sequences $x=(x_1, x_2, \dots)$,
with $x_k\in \{0,1, 2, \dots , r-1\}$ for all $k\in \N$. The map $$ \begin{aligned}
 X\rightarrow &[0, 1] \\
 x\mapsto& \sum_{k=1}^{\infty} x_k r^{-k}
\end{aligned} $$ is a bijection off a countable set. Under it, the Lebesgue measure can be seen as the product measure given by $\prod_k \mu$,
where $\mu(a_j) =\frac{1}{r}$, for all $a_j\in \{0,1, \dots, r-1\}$. By Proposition \ref{medida producto} this mesure is $\pi$-quasi-invariant,
for all action $\pi$ of $\B_\infty$ on $X\approx[0,1]$ with the following property: each $\pi_k$ acts changing only finite coordinates $x_j$ of
$x$. \end{exa}

\subsection{Ergodic Measures} We start with the definition. \begin{defn} Let $\mu$ be a $\pi$-quasi-invariant measure, $\mu$ is said
{\it{$\pi$-ergodic}} if each subset $E$ of $X$ invariant by the action of $\pi$ have measure zero or its complement have measure zero.

Equivalently (see \cite {J} Proposition 11.1.2, or \cite {G} p. 15), $\mu$ is $\pi$-ergodic if every essentially bounded measure function
$\varphi(x)$ on $X$ which is invariant by the action of $\pi$ (that is $\varphi(x) =\varphi(\pi_k(x))$ for almost all $x\in X$ and all
$k$,$1\leq k \leq n-1$, $n\leq \infty$), is constant for almost all $x\in X$. \end{defn}

\begin{defn} Two measures $\mu$ and $\mu'$ defined over the $\sigma$-algebra of subsets of $X$, with $\mu(X)=\mu'(X)=1$, are {\it{disjoint}} if
there is a measurable set $F$ such that $\mu(F)=1$ and $\mu'(F)=0$. \end{defn}

\begin{prop} Let $\mu$ and $\mu'$ be two $\pi$-ergodic measures on $X$ such that $\mu(X)=\mu'(X)=1$. If $\mu'$ is absolutely continuous respect
to $\mu$, then $\mu$ and $\mu'$ are equivalent. If they are not equivalent, then they are disjoint. \end{prop}

\begin{proof} The second statement implies the first one. Assume that $\mu$ and $\mu'$ are not equivalent, then there is a measurable set $E$
with $\mu(E)>0$ and $\mu'(E)=0$. As $\mu'$ is $\pi$-quasi-invariant, $\mu'(\pi_k(E))=0$ for all $k=1, \dots, n-1$. Furthermore,
$\mu'(\pi(\tau)(E))=0$ for all $\tau \in \B_n$. Therefore $F=\bigcup_{\tau\in \B_n} \pi(\tau)(E)$ is a measurable set with $\mu'(F)=0$.

On the other hand $\mu(F)\geq \mu(E)>0$, then we have that $\mu(F)=\mu(X)=1$, since $F$ is invariant by the action $\pi$ and $\mu$ is ergodic.
Therefore $\mu$ and $\mu'$ are disjoint. \end{proof}

\begin{rem} From this proposition on deduce easily that every $\pi$-ergodic discrete measure is concentrated in the orbit of an element $x\in
X$, that is, the points in the orbit of $x$ are the unique points of $X$ of non zero measure. \end{rem}


\section{Irreducibility and Factor Representations}\label{irreduc}

Given the $5$-tuple $(\pi, X, \mu, \nu,U)$, we will analyze the irreducibility of the associated representation $\phi_{(\pi, X, \mu, \nu,U)}$.

\begin{prop} Let $\phi_{(\pi, X, \mu, 1, U)}$ be a representation of $\B_n$ such that for all $k$, $\phi_k:=\phi(\tau_k)$ is a self-adjoint
operator.
 If the measure $\mu$ is $\pi$-ergodic, $\mu(X)=1$ and $U(\tau_k, \pi_k^{-1}x)$ is not a
constant operator, then the representation is irreducible. \end{prop}

\begin{proof} Let $\mathcal{K}\subseteq \CH$ be a closed invariant subspace, then by corollary \ref{subespacio}, $\widetilde{\phi}:=P_\CK \phi
P_\CK = \phi_{(\pi_\CK, X_\CK, \mu_\CK, \nu_\CK, P_\CK U P_\CK)}$, where  $X_\CK \subseteq X$, $\mu_\CK=\mu_{|_{X_\CK}}$ and $\nu_\CK(x)\leq
\nu(x)$.

As $\nu(x)=1$ for almost all $x\in X$ and $\CK_x\subseteq \CH_x$, then $\CK_x=0$ or $\CK_x=\CH_x$, for $x\in X_\CK$. Now, let $A=\{x\in X_\CK :
 \CK_x\neq 0\}=\{x\in X : \nu_{\CK}(x)=1\}\subseteq X$. It is a measurable set because $\nu_{\CK}$ is measurable.
 Moreover, it is invariant by $\pi$, since $\nu_\CK$ is $\pi$-invariant. As $\mu$ is $\pi$-ergodic, $\mu(A)=0$ or $\mu(A)=1$. Then $\CK=0$ or
 $\CK=\CH$.
\end{proof}


\begin{defn} A representation $\phi$ is a {\it{factor representation}} if the von Neumann algebra $M$ generated by $\{\phi(\tau_k)\}_k$ is a
factor, that is, if the commutator $M'$ contains just the operators $\lambda 1_M$, with $\lambda \in \C$. \end{defn}

If $M$ is the von Neumann algebra generated by an irreducible representation $\phi$, then by Schur's Lemma, $M'=\C 1_{\CH}$, that is $\phi$ is a
factor representation. The converse is only true if the representation is unitary.

From  von Neumann algebras theory it is known that every representation can be decomposed, in a unique way, as sum or direct integral of factor
representations.

\begin{prop} If $\phi_{(\pi, X, \mu, \nu, U)}$ is a factor representation, then $\mu$ is $\pi$-ergodic. \end{prop}

\begin{proof} Let $\varphi$ be an essentially bounded measurable function such that $\varphi(x)=\varphi(\pi_k x)$. By \ref{condicion (3)} we
have that $$ \begin{array}{rl} (\phi_k M_{\varphi}f)(x)= &(M_{\varphi \circ \pi_k^{-1}}\phi_k f)(x)=\varphi(\pi_k^{-1}x)(\phi_kf)(x)\\
                     =  &\varphi (\pi_k^{-1}x) U(\tau_k, \pi_k^{-1}x) \sqrt{\frac{d\mu
                     (\pi_k^{-1} x)}{d\mu(x)}} f(\pi_k^{-1}x)\\
                     = & \varphi (x) U(\tau_k, \pi_k^{-1}x) \sqrt{\frac{d\mu
(\pi_k^{-1} x)}{d\mu (x)}} f(\pi_k^{-1}x)\\
                     =& (M_{\varphi} \phi_k f)(x)
\end{array} $$ Then $M_{\varphi}\in M'$. This means that $M_{\varphi}=\lambda 1_{\CH}$, and $\varphi(x)=\lambda$ for almost all $x\in X$.
\end{proof}

\begin{cor} If $\phi:=\phi_{(\pi, X, \mu, \nu, U)}$ is a factor representation and $\nu$ is an essentially bounded measurable function, then
$\nu$ is constant. \end{cor}

\begin{proof} As $\phi$ is a factor representation, then $\mu$ is $\pi$-ergodic. This means that every $\pi$-invariant essentially bounded measurable function is a constant function. So is $\nu$.
 \end{proof}

Note that, under the condition of the corollary, $\nu(x)=a<\infty$ for all $x\in X$.

\begin{prop}\label{rep.irred} Let $\phi_{(\pi, X, \mu, \nu, U)}$ be an irreducible representation of $\B_n$, $n\leq \infty$, where the $5$-tuple
$(\pi, X, \mu, \nu, U)$ satisfies the conditions of the example ~\ref{sencillas} and $\mu$ is a $\pi$-invariant finite measure. If $\nu$ is an
essentially bounded measurable function then $\nu(x)=1$. \end{prop}

\begin{proof} As the representation is irreducible, it is a factor. Then, the measure is ergodic. Therefore $\nu(x)$ is constant. We suppose
that $\nu(x)=r<\infty$ for almost all $x$. By conditions of the example \ref{sencillas}, $U(\tau_k, \pi_k x)$ commute with $U(\tau_k, \pi_k
\pi_{k+1} \pi_k x)=U(\tau_{k+1}, \pi_{k+1} x)$ for all $k$, $1\leq k\leq n-2$. Hence, the operators $\{U(\tau_k, \pi_k x)\}_k$ commute.
Therefore, they simultaneously triangularize for each $x$. Then there is $v_x\in \CH_r$ such that $U(\tau_k, \pi_k x)v_x=\lambda_k (x) v_x$, for
every $k$.  We choose an element in each class of $X/\B_n$, $y_1, y_2, \dots $, and unitary vectors $v_{y_1}, v_{y_2}, \dots \in \CH_r$ such
that $U(\tau_k, \pi_k y_l)v_{y_l}=\lambda_k(y_l) v_{y_l}$, for each $l=1,2, \dots$. We define the vector function $x\rightarrow f(x)$ such that
$f(y)=v_{y_l}$ if $y=\pi(\tau)(y_l)$, for some $\tau\in \B_n$. Then, $f\in \CH$ because $\int_X\abs{f(x)}^2d\mu(x) = \mu(X)<\infty $. It
 verifies that
$$ (\phi_kf)(x)=\lambda_k (x) f(x) $$

Then, since $f$ is non zero,  the closed subspace $\mathcal{K}$ generated by $f$ is a non zero invariant subspace. Moreover, $\CK$ is invariant
by $\phi_k^*$ for all $k$, because $$ (\phi_k^*f)(x)=U(\tau_k, x)^*f(\pi_k x)=\overline{\lambda_k (x)} f(x) $$ Thus, by corollary
\ref{subespacio}, $\CK=\int_{X_\CK} \CK_x d\mu_\CK(x)$. Where $\nu_\CK(x)=1$ for almost all $x\in X_\CK$.

If $r=\nu(x)>1$, there is $w_{y_l}\in \CH_{y_l}$, linearly independent  to $v_{y_l}$. Let $\mathcal{\widetilde{K}}= \int_X^{\oplus} \C w_x\
d\mu$, where $w_x:=w_{y_l}$ if $x=\pi(\tau) (y_l)$, for some $\tau \in \B_n$. $\mathcal{\widetilde{K}}$ is a closed subspace of $\CH$ such that
$\widetilde{\mathcal{K}}\cap \mathcal{K}= 0$ because $\widetilde{\mathcal{K}}_x\cap \mathcal{K}_x = 0$ for almost all $x\in X$. Therefore
$\mathcal{K}\neq \CH$ which is a contradiction since the representation is irreducible. Then $\nu(x)=1$.
\end{proof}

\begin{prop} Let $\rho$ be
an irreducible self-adjoint representations of $\B_n$ on a  vector space $V$ of dimension $m\leq \infty$. Let $(\pi, X, \mu, \nu, U)$ be a
$5$-tuple where $\mu$ is a $\pi$-ergodic invariant measure, $\nu(x)=m$ and $U(\tau_k, x):=\rho (\tau_k)$ for almost all $x\in X$. Then
$\phi_{(\pi, X, \mu, \nu, U)}$ is an irreducible representation of $\B_n$.
\end{prop}

\begin{proof} Let $\CK$ be an invariant closed subspace of $\CH:=\int_X V d\mu(x)$. As $\phi_{(\pi, X, \mu, \nu, U)}$ is self-adjoint, by
Corollary  \ref{subespacio}, $\CK=\int_{X_\CK} \CK_x d\mu_\CK(x)$, where $X_\CK\subset X$, $\CK_x$ is a subspace of $V$ and
$\mu_\CK=\mu_{|_{X_\CK}}$. Given $y\in X_\CK$, we have to see that $\CK_y$ is an invariant subspace by $\rho$. Let $v\in \CK_y$ be a non zero
vector, let us see that for all $k$, $1\leq k\leq n-1$, $\rho(\tau_k)v \in \CK_y$. For each $k$, let $f_k\in \CK$ defined by $$ f_k(x)=\left\{
\begin{array}{ll}
              v        &  \text{    if    } \pi_k x=y \\
              0        & \text{    if    }  \pi_k x\neq y
\end{array}\right. $$ Then, as $\CK$ is invariant by $\phi$, we have that
$$ (\phi_k f_k)(x)= U(\tau_k, \pi_k^{-1} x) f_k(\pi_k^{-1} x)=\left\{
\begin{array}{ll}
              \rho(\tau_k)v        &  \text{    if    } \pi_k \pi_k^{-1} x=y \\
              0                   & \text{    if    }  \pi_k \pi_k^{-1} x\neq y
\end{array}\right. $$
But $ (\phi_k f)(x)\in \CK_x$ for almost all $x$ since $\CK$ is invariant. Hence, $\rho(\tau_k)v \in \CK_y$ as we wanted.
Therefore $\CK_x$ is invariant for  almost all $x\in X_\CK$, then $\CK_x=0$ or $V$ because $\rho$ is irreducible. Now, let $A=\{x\in X_\CK :
\CK_x= 0\}=\{x\in X : \nu_{\CK}(x)=0\}\subseteq X$. It is a measurable set because $\nu_{\CK}$ is measurable. Moreover, it is
invariant by $\pi$, since $\nu_\CK$ is $\pi$-invariant. As $\mu$ is $\pi$-ergodic, $\mu(A)=0$ or $\mu(A)=1$. Then $\CK=\CH$ or $\CK=0$.
\end{proof}

\begin{rem} There are representations of $\B_n$ ($n\leq\infty$), that  are not completely reducible.

We construct the following example using Theorem \ref{teorema 1}. Let $X=\{(x_1, \dots, x_n): x_i=0, 1\}$, $\pi_k$ acting on each $x\in X$ by
permutation of its coordinates $x_k$ and $x_{k+1}$, $\nu(x)=2$ and $U(\tau_k, x)=A$ for almost all $x\in X$ and for all $k$, where $A$ is the
following matrix $$ A= \left(\begin{array}{cc}
            1    &1   \\
            0    &1
\end{array}\right) $$ The operator $U(\tau_k, \pi_k x)$ satisfy the relations of the example \ref{sencillas}. If $f(x)\equiv (1, 0)$ then the
closed subspace $\mathcal{K}$ generated by $f$ is invariant. Moreover, it is invariant by $\phi_k^*$, then, by corollary \ref{subespacio}, $\CK=
\int_{X_\CK} \CK_x d\mu_{\CK}(x)$. But $\CK$ has not invariant complements. Indeed, suppose that there exits a closed subspace
$\mathcal{W}\subset \CH$, invariant by $\{\phi_k\}$, such that $\mathcal{W}\oplus \mathcal{K} =\CH$. We may suppose that $\CW$ is orthogonal to
$\CK$, in the contrary, we change $\CW$ by $\widetilde{\CW}=\{w-P_{\CK}w : w\in \CW\}$, where $P_{\CK}$ is the orthogonal projection over $\CK$.
Then $\widetilde{\CW}$ is an orthogonal invariant complement of $\CK$.

Moreover $\CW$ is invariant for $\phi_k^*$, then $\phi_k$ commute with $P_{\CW}$, the orthogonal projection over $\CW$. Therefore
$\widetilde{\phi_k}:=P_{\CW} \phi_k P_{\CW}$ satisfy the hypothesis of the Theorem \ref{teorema 2}. Hence, $\mathcal{W}=\int_X \mathcal{W}_x
d\mu(x)$, with $\mathcal{W}_x \subset \CH_x$ and $\mathcal{W}_x \oplus \mathcal{K}_x=\CH_x$ for almost all $x$. Then $\dim \mathcal{W}_x=1$ for
almost all $x$, since $\nu(x)=2$ and $\dim \mathcal{K}_x=1$. This is a contradiction since $A$ has a unique eigenspace of dimension $1$.
\end{rem}

\end{document}